\documentclass[A4paper,11pt]{amsart}
\usepackage[utf8]{inputenc}
\usepackage[english]{babel}
\usepackage{multirow}
\usepackage{anysize}
\usepackage{color}
\usepackage{multicol}
\usepackage[colorlinks=true,
            linkcolor=blue,
            citecolor=blue,
            urlcolor=blue]{hyperref}

\usepackage{array} 
\usepackage{lscape}
\usepackage{tikz-cd}
\usepackage{float}
\usepackage{standalone}
\usepackage{subcaption}
\usepackage{tikz}
\usepackage{pgfplots}
\usetikzlibrary{decorations.markings, calc, intersections, external, backgrounds, cd, fillbetween, arrows.meta}
\tikzset{math to/.tip={Glyph[glyph math command=rightarrow]}}
\tikzcdset{arrow style=tikz,
           diagrams={>=stealth}
           }
\newcommand{\Z}{\mathbb{Z}} %Numeros Enteros
\usepackage{multirow, array}
\usepackage{amsmath}

\DeclareMathOperator{\Sym}{Sym}

\DeclareMathOperator{\Pic}{Pic}

\DeclareMathOperator{\mult}{mult}

\DeclareMathOperator{\Nm}{Nm}
\DeclareMathOperator{\ima}{Im}

\DeclareMathOperator{\lcm}{lcm}

\DeclareMathOperator{\type}{type}
\DeclareMathOperator{\Hom}{Hom}

\DeclareMathOperator{\ptype}{type}

\usepackage{enumitem}
\usepackage{amsfonts}
\usepackage{amsmath,calligra}
\usepackage{amssymb}
\usepackage[arrow, matrix, curve, xdvi, dvips]{xy}

\usepackage{blkarray}
\usepackage{booktabs}
\usepackage{mathtools}
\usepackage{mathrsfs}
\usepackage{times}

\usepackage{graphicx}
\usepackage[left=2cm,right=2cm, top=3cm, bottom=3cm]{geometry}
\usepackage[usenames,dvipsnames]{pstricks}
 \usepackage{pstricks-add}
 \usepackage{epsfig}

\usepackage{amsthm}
\newtheorem{theorem}{Theorem}[section]

\newtheorem{corollary}[theorem]{Corollary}

\newtheorem{lemma}[theorem]{Lemma}

\newtheorem{proposition}[theorem]{Proposition}
\newtheorem{remark}[theorem]{Remark}

\title{Prym--Tyurin varieties coming from intermediate coverings of curves}

\author{V\'{i}ctor Valdebenito Sep\'ulveda}

\address{Departamento de Matem\'atica y Estad\'istica, Universidad de La Frontera, Av. Francisco Salazar 01145 Temuco, Chile.}

\subjclass[2020]{14H40, 14H30, 14K10, 14K05.}

\thanks{This work has been supported by ANID BECAS/DOCTORADO NACIONAL 21221502}

\keywords{Abelian variety, Prym-Tyurin variety}

\email{victor.valdebenito@ufrontera.cl}

\begin{document}
\maketitle

\begin{abstract}
We introduce a new geometric construction of Prym--Tyurin varieties of odd prime exponent $p$ arising from intermediate coverings of $\mathbb Z_p\times\mathbb Z_p$--Galois covers of smooth curves. We determine precisely when these coverings give rise to Prym--Tyurin varieties, namely when the base curve is elliptic and in the case of isotropic \'etale covers over curves of genus $2$. The explicit nature of the construction makes it possible to analyze the geometry of the associated Prym--Tyurin varieties and to prove that the corresponding Prym--Tyurin maps, indexed by the choice of a generating vector, are generically finite onto their images.
\end{abstract}

\section{Introduction}

A classical approach to the study of the moduli space $\mathcal A_g$ of principally polarized abelian varieties (ppavs) is through the geometry of curves. Two fundamental examples arising in this way are Jacobian varieties and classical Prym varieties. This perspective is particularly effective in low dimensions, where the rich theory of curves allows one to describe large portions of $\mathcal A_g$ in geometric terms. Indeed, it is well known that for $g\le 3$ the Torelli map $t:\mathcal M_g\to \mathcal A_g$ is dominant, so that the Jacobian locus is dense in $\mathcal A_g$. Moreover, for $g\le 5$ a general ppav is a classical Prym variety, as the Prym map $\mathcal P:\mathcal R_{g+1}\to \mathcal A_g$ is dominant.

Motivated by this picture, Beauville \cite{beauville_survey} raised the problem of finding a stratification
$$
\mathcal J_g \subset \overline{\mathcal P}_g \subset \cdots \subset \mathcal A_g
$$
by geometrically defined subvarieties, extending the role played by Jacobians and Prym varieties. This led to the introduction of Prym--Tyurin varieties, first studied by Tyurin \cite{tyurin}, as abelian subvarieties of Jacobians whose induced polarization is a multiple of a principal polarization.

More precisely, a principally polarized abelian variety $(P,\Xi)$ is called a Prym--Tyurin variety of exponent $\varepsilon$ if there exists a smooth projective curve $C$ and an embedding $P\hookrightarrow JC$ such that the restriction of the canonical polarization satisfies $\Theta_P \equiv \varepsilon\,\Xi$.

Prym--Tyurin varieties provide a natural extension of Jacobians and classical Pryms. By the Matsusaka--Ran theorem \cite[11.8.1]{BL}, exponent $1$ corresponds exactly to Jacobians, while Welters showed in \cite{Welters} that exponent $2$ recovers classical Prym varieties. In the same work, Welters proved that every ppav is a Prym--Tyurin variety of some exponent. Moreover, combining Welters’ criterion with the geometry of curves on the Kummer variety, one obtains the general bound that any ppav of dimension $g$ is a Prym--Tyurin variety of exponent at most $2^{g-1}(g-1)!$.

Despite this general existence result, a central problem remains to construct Prym--Tyurin varieties of small exponent, and in particular of exponent at least $3$, far below the bound given by Welters. Only a limited number of explicit examples are known. Constructions due to Kanev \cite{kanev}, based on correspondences, as well as subsequent generalizations \cite{LRR,Salomon,CLRR}, produce important families, but remain restricted in scope. More recently, examples of Prym--Tyurin varieties of small exponent have been obtained as isotypical components of Jacobians of curves with group action \cite{LR_isotypical}, using in part computational techniques.

In this work, we introduce a new geometric construction of Prym--Tyurin varieties based on intermediate coverings. This approach does not rely on correspondences and allows us to produce explicit families of Prym--Tyurin varieties of odd prime exponent $p$, far below the general bound predicted by Welters.

More precisely, let $p$ be an odd prime and let $f:\widetilde C \to C$ be a Galois covering with group $G_p\simeq \mathbb Z_p \times \mathbb Z_p$. Fix a subgroup $\langle g \rangle \subset G_p$ of order $p$, set $C_g=\widetilde C/\langle g \rangle$, and denote by $k_g:\widetilde C \to C_g$ and $h_g:C_g \to C$ the induced intermediate coverings, so that $f = h_g \circ k_g$. We consider the Prym variety $P(h_g)$ and its pullback
$$
P = k_g^*P(h_g) \subset J\widetilde C.
$$

Our first main result gives a geometric criterion ensuring that the associated abelian subvariety $P$ is a Prym--Tyurin variety of exponent $p$.

\begin{theorem}
Let $p$ be an odd prime. For a $G_p$--Galois covering $f:\widetilde C \to C$ satisfying the assumptions of Section \ref{section4}, the associated abelian subvariety $P = k_g^*P(h_g)$ is a Prym--Tyurin variety of exponent $p$.
\end{theorem}

In Section \ref{section4} we determine the geometric situations in which this construction produces Prym--Tyurin varieties of exponent $p$. More precisely, we prove that this happens for $G_p$--coverings over elliptic curves with prescribed ramification, and also for isotropic étale coverings over curves of genus $2$. By contrast, the obstruction established in Section \ref{section3} shows that, when the base curve has genus at least $3$, this construction never produces a Prym--Tyurin variety, whether $h$ is étale or ramified.

In particular, this construction produces families of Prym--Tyurin varieties parametrized by the corresponding Hurwitz spaces of $\mathbb Z_p \times \mathbb Z_p$--covers.

A key feature of this construction is its explicit geometric nature, which makes it possible to analyze the geometry of the associated Prym--Tyurin varieties in detail. We focus on the case in which the base curve is an elliptic curve $E$. Let $h:C\to E$ be a cyclic covering of degree $p$ ramified at $m$ points. Fixing a $p$--torsion point $\xi\in E[p]$, we construct an \'etale covering $k:\widetilde C\to C$, which produces the Prym--Tyurin variety $P=k^*P(h)$.

For each choice of generating vector associated with the cyclic covering $h:C\to E$ of degree $p$ with the prescribed ramification data, this construction naturally defines a Prym--Tyurin map
$$
\mathcal{PT}_{p,m}:\mathcal R(p,m)^\xi\longrightarrow \mathcal A_{\frac{m}{2}(p-1)},
$$
which assigns to a pair $(h,\xi)$ the associated Prym--Tyurin variety. Our second main result concerns the geometry of these maps.

\begin{theorem}
For every odd prime $p\ge 5$, the Prym--Tyurin map $\mathcal{PT}_{p,m}$ is generically finite onto its image for $m\ge 2$; and for $p=3$, the map is generically finite for $m\ge 4$.
\end{theorem}

The result is a consequence of the generic finiteness of the Prym--map of elliptic ramified covers, which can be described in terms of multiplication maps of sections on the base elliptic curve. We use the notation $\mathcal{PT}_{p,m}$ after fixing the discrete choice of generating vector.

The paper is organized as follows. In Section \ref{section2} we recall background on polarized abelian varieties, coverings of curves, and Prym varieties. Section \ref{section3} develops the Prym--Tyurin construction via intermediate coverings and establishes the criterion for exponent $p$. Section \ref{section4} applies this framework to $\mathbb Z_p\times\mathbb Z_p$--covers and proves the existence of the corresponding families. Finally, Section \ref{sec5} studies the associated maps and proves their generic finiteness.

\subsection*{Acknowledgments}
The author is deeply grateful to Rubí E. Rodríguez for her guidance throughout the development of this work, which forms part of his doctoral thesis. He also thanks Ángel Carocca and Robert Auffarth for their invaluable insights, which greatly enriched this research, and Juan Carlos Naranjo for his hospitality in Barcelona, where part of this work was carried out, and for many fruitful discussions.

\section{Preliminaries}\label{section2}

In this section we recall some known results and fix notation concerning abelian varieties and Galois covers of curves, which will be used throughout the paper.

\subsection{Abelian varieties and polarizations} Let $(A, L)$ be a polarized abelian variety with associated isogeny $\phi_L : A \to \widehat{A}$ where $\widehat A$ denotes the dual variety of $A$. For any subvariety $B$ of $A$, denote by $\imath_B:B\hookrightarrow A$ the canonical embedding,  $L_B$ the induced polarization on $B$ by $L$, $e_B:=\exp K(L_B)$ the exponent of the induced polarization on $B$, $B[n]$ the $n$--torsion points of $B$,  and denote by $K(L)$ the kernel of the isogeny $\phi_L$. For any polarization $L$ on an abelian variety $A$ of dimension $g$ it holds that
$$K(L)\simeq (\mathbb{Z}_{d_1})^2\times\dots\times (\mathbb{Z}_{d_g})^2,$$

where the integers $d_1,\dots,d_g\geq 1$ satisfy that $d_j \mid d_{j+1}$ for $j\in\{1,\dots,g-1\}$. The tuple of positive integers $(d_1,\dots,d_g)$ is called the {\it type} of the polarization, sometimes denoted by $\type (L)$. In the case that $\type(L)=(1,\dots,1)$, we have that $\phi_L$ is an isomorphism and we say that $L$ is a principal polarization of $A$.

Recall that for any abelian subvariety $B$ of $A$ there exists the norm endomorphism denoted by $N_B$, given by the composition
$$
\begin{tikzcd}
 A\ar[r,"\phi_L"] &\widehat{A}\ar[r,"\hat{\imath}_B"]&\widehat{B}\ar[r,"\psi_{L_B}"]& B\ar[r,"\imath_B"]& A   
\end{tikzcd}
$$
where $\psi_{L_B}=e_B\cdot\phi_{L_B}^{-1}:\widehat{B}\to B$ and $\phi_{L_B}^{-1}$ is the inverse isogeny of $\phi_{L_B}$ in $\Hom(\widehat{B},B)\otimes \mathbb{Q}$.

In the following, let $(B,P)$ be a pair of abelian subvarieties of $A$, complementary with respect to the polarization $L$. Let  $\mu:B\times P\to A$ denote the addition map. The kernel of $\mu$ is given by

$$\ker \mu=\{(x,-x)\in B\times P\mid x\in B\cap P\}\simeq B\cap P.$$

We have the following useful results.

\begin{proposition}[{\cite[Proposition 5.3.11]{BL}}]\label{prop:kernel}
$\ker \mu\subset B[e_B]\cap P[e_P]$.
\end{proposition}

\begin{lemma}[{\cite[Lemma 2.3]{Lange-Pauli}}]\label{lemma:kernel}
With the notations above
$$|K(L_B)|\cdot |K(L_P)|=|B\cap P|^2\cdot |K(L)|.$$
\end{lemma}

\begin{remark}
In \cite{Lange-Pauli} the authors assume that $e_B=e_P$, however, from the proof one can conclude that this assumption is not necessary in Lemma \ref{lemma:kernel}. Moreover, one can obtain a more general result as follows.
\end{remark}

\begin{lemma}\label{lemma:kernel_mu}
Let $(A,L)$ be a polarized abelian variety and let $B_1,\dots,B_r$ be abelian subvarieties of $A$ with norm endomorphisms pairwise orthogonal (i.e.\ $N_{B_i}N_{B_j}=0$ for $i\neq j$), and assume that the addition map 
$$\mu:B_1\times \dots\times B_r\to A$$
is an isogeny. Then the induced polarization $\mu^*L$ splits as a product polarization and
$$|K(L_{B_1})|\cdots|K(L_{B_r})|=|\ker \mu|^2\cdot|K(L)|.$$
\end{lemma}

\begin{proof}
Consider the polarization on $B_1\times\dots\times B_r$ induced by pullback via the addition map $\mu$. We have that
\begin{equation}\label{eq:mu}
|\ker (\phi_{\mu^*L})|=|\ker(\widehat{\mu}\phi_L \mu)|=|\ker \mu|^2\cdot |K(L)|.    
\end{equation}

On the other hand, consider the isogeny
$$\phi_{\mu^*L}=\phi_{(\imath_{B_1}+\dots+\imath_{B_r})^*L}=(\widehat{\imath}_{B_1}+\dots+\widehat{\imath}_{B_r})\phi_L\,(\imath_{B_1}+\dots+\imath_{B_r}),$$
which, as a map $B_1\times\dots\times B_r\to \widehat{B}_1\times\dots\times \widehat{B}_r$, is given by the matrix
$$(a_{ij})=\widehat{\imath}_{B_i}\phi_L \imath_{B_j}.$$

Since $N_{B_i}N_{B_j}=0$ for all $i\neq j$ we obtain
$$0=N_{B_i}N_{B_j}=\imath_{B_i}\psi_{L_{B_i}}(\widehat{\imath}_{B_i}\phi_L \imath_{B_j})\psi_{L_{B_j}}\widehat{\imath}_{B_j}\phi_L.$$
Since $\psi_{L_{B_j}}\widehat{\imath}_{B_j}$ is surjective onto $B_j$, $\imath_{B_i}$ is injective and $\psi_{L_{B_i}}$ and $\phi_L$ are isogenies, we conclude that the map $\widehat{\imath}_{B_i}\phi_L \imath_{B_j}=0$. Therefore $\mu^*L$ splits, and consequently

$$\phi_{\mu^*L}=\phi_{L_{B_1}}\times\dots\times \phi_{L_{B_r}}.$$

Then we have that

\begin{equation}\label{eq:mu*}
|K(\mu^*L)|=|\ker (\phi_{L_{B_1}}\times\dots\times \phi_{L_{B_r}})|=|K(L_{B_1})|\cdots |K(L_{B_r})|.
\end{equation}
Combining \eqref{eq:mu} and \eqref{eq:mu*} we obtain the result.
\end{proof}

\begin{corollary}\label{cor:transitivity}
In the situation of Lemma \ref{lemma:kernel_mu}, fix an index $i$ and set $P$ to be the complementary abelian subvariety of $B_i$ in $A$ with respect to $L$. Then $P$ is isogenous to the product
$$
\prod_{j\neq i} B_j,
$$
and the polarization induced by $L$ on $P$ is the product polarization $\otimes_{j\neq i} L_{B_j}$. In particular, for every $j\neq i$ the polarization induced by $L$ on $B_j$ coincides with the polarization induced by $L|_P$ on $B_j$.
\end{corollary}

\begin{proof}
By Lemma \ref{lemma:kernel_mu}, the pullback $\mu^*L$ splits as the product $L_{B_1}\otimes\cdots\otimes L_{B_r}$. Hence the orthogonal complement $P$ corresponds, up to isogeny, to the product of the remaining factors, endowed with the product polarization, and the statement follows by restriction.
\end{proof}

The following relation between the exponents of complementary abelian subvarieties helps to determine the type of the induced polarization in some cases.

\begin{lemma}\label{lemma:exponents}
Let $(A,L)$ be a polarized abelian variety and let $B,P\subset A$ be complementary with respect to $L$, with addition isogeny $\mu:B\times P\to A$. Then the exponents satisfy
$$ e_P\ \mid \exp K(L)\exp(\ker\mu) \quad and \quad \exp(\ker\mu) \mid \lcm(e_B,e_P).$$
\end{lemma}

\begin{proof}
Let $e_L = \exp K(L)$ and $e_\mu = \exp(\ker\mu)$.  For any $x \in K(L_P)$, we have $(0,x) \in K(\mu^*L)$, so
$$
0 = \phi_{\mu^*L}(0,x) = \hat\mu(\phi_L(x)) \implies \phi_L(x) \in \ker \hat\mu.
$$
Since $\ker\mu\simeq \ker \hat\mu$ it holds that $\exp(\ker\hat\mu) = e_\mu$, and it follows that $e_\mu\,\phi_L(x) = 0$, hence $e_\mu\,x \in K(L)$.  By definition of $e_L$, we have $e_L(e_\mu\,x) = 0$.  

For the second divisibility, note that $\ker\mu \subset K(\mu^*L) = K(L_B)\times K(L_P)$. The exponent of a subgroup divides the exponent of the group, hence
\begin{equation*}
\exp(\ker\mu) \mid \lcm(e_B,e_P).\qedhere
\end{equation*}
\end{proof}

Recall also the Weil pairing associated with a polarization. If $A=V/\Lambda$ and $H$ is the Hermitian form representing $c_1(L)$, we denote the Weil pairing of $L$ by
$$
e^L:K(L)\times K(L)\longrightarrow \mathbb C^*.
$$
It is defined by
$$
e^L(x,y)=\exp\bigl(-2\pi i\,\operatorname{Im}H(\tilde{x},\tilde{y})\bigr),
$$
where $\tilde{x},\tilde{y}\in \Lambda(L):=\pi^{-1}(K(L))$ are lifts of $x,y\in K(L)$. This pairing is a useful tool for controlling exponents of polarizations descending to quotients. For a subgroup $G\subset K(L)$ we write
$$
G^{\perp_L}:=\{x\in K(L)\mid e^L(x,g)=1 \text{ for all } g\in G\}.
$$
Thus $G$ is isotropic if $G\subset G^{\perp_L}$, and maximal isotropic if $G=G^{\perp_L}$.

\begin{proposition}\label{prop:exponent_descended_polarization}
Let $(A,L)$ be a polarized abelian variety and let $G\subset K(L)$ be a finite subgroup. Let $\pi:A\to A/G$ be the quotient isogeny. Suppose that $L$ descends to a polarization $M$ on $A/G$, so that $\pi^*M\simeq L$. Then, for every integer $n\geq 1$,
$$
\exp K(M^{\otimes n})=n
$$
if and only if $G$ is maximal isotropic in $K(L)$ with respect to $e^L$.
\end{proposition}

\begin{proof}
By \cite[Corollary 4.1.1]{MoretBailly}, the kernel of the descended polarization is $K(M)\simeq G^{\perp_L}/G$. Hence $M$ is principal if and only if $G^{\perp_L}=G$. If $M$ has type $(d_1,\ldots,d_g)$, then $M^{\otimes n}$ has type $(nd_1,\ldots,nd_g)$, and therefore $\exp K(M^{\otimes n})=n\exp K(M)$. The result follows.
\end{proof}

\subsection{Galois coverings and Prym varieties}

We now turn to Galois covers and Prym varieties. We recall some known results, following \cite{LR2022}. Let $f : \widetilde{C} \to C$ be a finite cover of smooth projective curves. The ramification and branch divisors are defined by 
$$
R_f:=\sum_{p\in\widetilde C}\bigl(\mult_p(f)-1\bigr)\,p,
\quad
B_f:=\sum_{q\in C}\left(\sum_{p\in f^{-1}(q)}\bigl(\mult_p(f)-1\bigr)\right)q,
$$
where $\mult_p(f)$ denotes the multiplicity of $f$ at the point $p\in\widetilde C$. The ramification degree of $f$ is $r(f):=\deg R_f$.

We say that a cover $f : \widetilde{C} \to C$ is a {\it Galois cover with group $G$} if $C$ is the quotient $\widetilde{C}/G$ and $f$ is the projection map. Let $B(f)=\{q_1,\dots,q_s\}$ be the branch points of a Galois cover $f$. For each $i$, choose a point $p_i \in f^{-1}(q_i)$ and denote by $G_i := G_{p_i}$ its stabilizer, of order $n_i := |G_i| \ge 2$, which is always a cyclic subgroup of $G$. The \emph{signature} of the cover is the tuple $(g_C; n_1,\dots,n_s)$, where $g_C$ is the genus of $C$.

For a Galois cover of signature $(g_C; n_1,\dots,n_s)$, the Hurwitz formula reads
\begin{equation}\label{eq:RH_formula}
g_{\widetilde{C}} = |G|(g_C - 1) + 1 + \frac{|G|}{2} 
\sum_{i=1}^{s} \left(1 - \frac{1}{n_i}\right).
\end{equation}

where $g_{\widetilde{C}}$ denotes the genus of $\widetilde{C}$. Let $g_C, s \ge 0$, and let $n_1,\dots,n_s \ge 2$ be integers. A $(2g_C+s)$-tuple
$$
(a_1,\dots,a_{g_C},b_1,\dots,b_{g_C};c_1,\dots,c_s)
$$
of elements of $G$ is called a \emph{generating vector of type $(g_C;n_1,\dots,n_s)$} if the following conditions are satisfied:
\begin{enumerate}
    \item[(a)] The elements $a_1, \dots, a_{g_C}, b_1, \dots, b_{g_C}, c_1, \dots, c_s$ generate the group $G$.
    \item[(b)] $c_i$ is of order $n_i$ for $i = 1, \dots, s$.
    \item[(c)] $\prod_{i=1}^{g_C} [a_i, b_i] \prod_{j=1}^{s} c_j = 1$, where $[a_i, b_i] = a_i b_i a_i^{-1} b_i^{-1}$.
\end{enumerate}
The following result is a consequence of the Uniformization Theorem (see \cite{B1991}).

\begin{theorem}\label{theo:Riemann:existence}
Given a finite group $G$, there exists a curve $\widetilde{C}$ of genus $g_{\widetilde{C}}$ endowed with a $G$-action of signature $(g_C;n_1,\dots,n_s)$ if and only if:
\begin{itemize}
    \item[(i)] the Hurwitz formula \eqref{eq:RH_formula} holds;
    \item[(ii)] $G$ admits a generating vector of type $(g_C;n_1,\dots,n_s)$.
\end{itemize}
\end{theorem}

Let $f:\widetilde C\to C$ be a finite morphism of smooth projective curves, and denote by $(J\widetilde C,\widetilde\Theta)$ and $(JC,\Theta)$ the corresponding principally polarized Jacobians. Viewing Jacobians as groups of degree--zero line bundles, the pullback of line bundles defines a homomorphism 
$$
f^*:\; JC \longrightarrow J\widetilde C.
$$

This homomorphism is not injective if and only if $f$ factorizes through a cyclic \'etale cover $h:C'\to C$ of degree $\deg h\ge 2$, that is, if $f=h\circ k$ for some $k:\widetilde C\to C'$. Moreover, if $h$ is cyclic \'etale, then $\ker h^*$ has order $\deg h$ and is generated by the line bundle defining the cover $h$; in particular, if $k^*$ is injective, then $\ker f^*$ is cyclic of order $\deg h$ (see \cite[Propositions 3.2.2 and 3.2.3]{LR2022}).

On the other hand, the norm map $\Nm_f:\Pic^0(\widetilde C)\to \Pic^0(C)$ induces a homomorphism
$$
\Nm_f:\; J\widetilde C \longrightarrow JC.
$$

The \emph{Prym variety} of $f$ is defined as the abelian subvariety of $J\widetilde C$ complementary to $f^*JC$ with respect to the canonical polarization $\widetilde\Theta$, denoted by $P(f)=P(\widetilde C/C)$ or equivalently, $P(f)$ is the connected component of the origin of the kernel of the norm map. The following results are a useful tool to determine the type of $\widetilde\Theta_{P(f)}$ in some cases.

\begin{proposition}[{\cite[Proposition 3.2.8]{LR2022}}]\label{prop:type_prym}
Let $f:\widetilde C \to C$ be a cover of degree $d \ge 2$ and $g(C)=g\ge 1$. Then
\begin{enumerate}
\item[(a)]
If $f$ does not factorize via an \'etale cyclic cover of degree $\ge 2$, then $\widetilde\Theta_{P(f)}$ has type
$$
(1,\dots,1,d,\dots,d)
$$
with $g$ numbers equal to $d$ and $\widetilde g-2g$ numbers equal to $1$.

\item[(b)]
If $f$ factorizes as $f=h k$ with $h:C'\to C$ cyclic \'etale of degree $d_1<d$ and $k$ does not factorize via a cyclic \'etale cover of degree $\ge 2$, then $\widetilde\Theta_{P(f)}$ has type
$$
(1,\dots,1,\tfrac{d}{d_1},d,\dots,d)
$$
with $g-1$ numbers equal to $d$.
\end{enumerate}
\end{proposition}

\begin{proposition}[{\cite[Proposition 3.2.9]{LR2022}}]\label{prop:kernel_cover}
The homomorphism $f^*:JC\to J\widetilde C$ induces an isomorphism
$$
(\ker f^*)^{\perp_{d\Theta_C}} / \ker f^*
\longrightarrow 
K(\Theta_{f^*JC})
=
K(\Theta_{P(f)})
=
P(f)\cap f^*JC
\subset P[d],
$$
and
$$
|P(f)\cap f^*JC|
=
\frac{|JC[d]|}{|\ker f^*|^2}.
$$
\end{proposition}

\subsection{Construction of cyclic coverings and canonical decomposition}\label{section_coverings}

In order to analyze the codifferential of the Prym map in Section \ref{section5} and prepare the proof of the finiteness theorem, we briefly recall the basic theory of cyclic coverings of curves in full generality following Section 3.5 in \cite{EV}. The same results with a slightly different notation appear in \cite{Cornalba,Pa}.

A cyclic Galois cover of degree $n$ between smooth curves
$$
f:Y\to X
$$
is determined by the following data:
\begin{itemize}
\item a line bundle $L\in\mathrm{Pic}(X)$,
\item an effective divisor
$$
B=\sum_j b_jB_j,\; \text{with}\; 0<b_j<n \;\text{and}\; \sum b_j\equiv 0 \,(n),
$$
\end{itemize}
whose irreducible components are $B_j$, such that $L^{\otimes n}\simeq\mathcal O_X(B)$.  In particular,
$$
n\deg(L)=\deg(B)=\sum_j b_j.
$$

Consider the $\mathcal O_X$-algebra
$$
\mathcal A'=\bigoplus_{i=0}^{n-1}L^{-i},
$$
where the multiplication is defined using the identification $L^{-n}\simeq\mathcal O_X(-B)\subset\mathcal O_X$. Let $Y'$ be the spectrum of this algebra. Then $Y$ is the normalization of $Y'$.

For $0\le i\le n-1$ define
$$
L^{(i)}=L^i\left(-\left\lfloor\frac{i}{n}B\right\rfloor\right),
$$
where $\lfloor x\rfloor$ denotes the floor of $x$ and set
$$
\mathcal A=\bigoplus_{i=0}^{n-1}L^{-(i)}.
$$
By \cite[Claim 3.8]{EV}, the natural inclusions $L^{-i}\subset L^{-(i)}$ induce a morphism of $\mathcal O_X$-algebras $\phi:\mathcal A'\to\mathcal A$, and $\mathcal A$ carries a natural $\mathcal O_X$-algebra structure. Moreover, $Y$ is the spectrum of $\mathcal A$.

Fix a primitive $n$-th root of unity $\zeta$. The cyclic group $G=\mathbb Z/n\mathbb Z$, generated by $\sigma$, acts on $\mathcal A$ by $\mathcal O_X$-algebra automorphisms as follows: if $s$ is a local section of $L^{-(i)}$, then
$$
\sigma(s)=\zeta^i s.
$$
With this action one has $\mathcal A^G=\mathcal O_X$ and $Y/G=X$.

The fundamental structural decomposition is
\begin{equation}\label{structural}
f_*\mathcal O_Y
=
\bigoplus_{i=0}^{n-1}
L^{-(i)}
=
\bigoplus_{i=0}^{n-1}
L^{-i}\left(\left\lfloor\frac{i}{n}B\right\rfloor\right),
\end{equation}
and the summands $L^{-(i)}$ are precisely the eigensheaves of the $G$-action.

Since $f$ is finite, the functor $f_*$ is exact on coherent sheaves. Moreover, for finite morphisms between smooth curves one has the identification
$$
f_*\omega_Y\cong\mathcal Hom_{\mathcal O_X}(f_*\mathcal O_Y,\omega_X).
$$
As $f_*\mathcal O_Y$ is locally free of finite rank, no higher $\mathcal Ext$ terms appear. Substituting the decomposition \eqref{structural} into this formula yields
\begin{equation}\label{canonical}
f_*\omega_Y
=
\bigoplus_{i=0}^{n-1}
\omega_X\otimes L^{i}
\left(-\left\lfloor\frac{i}{n}B\right\rfloor\right).
\end{equation}

\section{Prym--Tyurin varieties coming from intermediate coverings}\label{section3}

As discussed in the introduction, some of the new examples of Prym--Tyurin varieties of small exponent obtained using the techniques of \cite{LR_isotypical} arise as isotypical components of Jacobians of curves admitting a group action. Although this approach produces genuinely new Prym--Tyurin varieties, these examples do not, in general, admit a direct geometric interpretation in terms of coverings of curves.

Motivated by this observation, the aim of the present section is to identify situations in which a Prym--Tyurin variety can be described geometrically as the Prym variety associated to an intermediate covering of $f:\widetilde{C}\to C$, endowed with the polarization induced from the Jacobian $J\widetilde C$. More precisely, we study factorizations of a cover $f:\widetilde{C}\to C$ of smooth projective curves as $f=h\circ k$, and analyze when the induced polarization on the abelian subvariety $k^{*}P(h)\subset J\widetilde C$ is a multiple of a principal polarization, giving rise to Prym--Tyurin varieties.

 Consider the diagram in Figure \ref{fig:intermediate_cover}.
 \begin{figure}[h!]
    \centering 
    \begin{subfigure}[t]{0.3\textwidth}
    \centering
    \begin{tikzcd}[row sep=large, column sep=large]
     &\widetilde{C} \ar[dd,"f"] \ar[dl,"k"]\\
     C' \ar[dr,"h"]& \\
     &C
    \end{tikzcd}

    \label{fig01}
    \end{subfigure} \begin{subfigure}[t]{0.4\textwidth}
    \centering
    \begin{tikzcd}[row sep=large, column sep=large]
     &J\widetilde{C} \\
     JC' \ar[ur,"k^*"]& \\
     &JC \ar[ul,"h^*"] \ar[uu,"f^*"]
    \end{tikzcd}

    \label{fig02}
    \end{subfigure} 
    \caption{}
    \label{fig:intermediate_cover}
\end{figure}
The following result describes the polarization induced on $f^*JC$, which will be a useful tool to determine in some cases the type of induced polarizations on certain abelian subvarieties.

\begin{lemma}\label{lemma_type}
Let $f:\widetilde{C}\to C$ be a cover of smooth projective curves over an elliptic curve $C$, which factors as $f=h\circ k$, where $k:\widetilde{C}\to C'$ is an \'etale cyclic cover of degree $p$ defined by  $\eta\in JC'[p]$, and $h:C'\to C$ is a ramified cover of degree $q$, with $p$ and $q$ odd primes. Then  the induced polarization on $f^*JC$ satisfies
$$
\ptype(\widetilde{\Theta}_{f^*JC})=
\begin{cases}
(q),  & \text{if } \eta \in h^*JC,\\
(pq), & \text{if } \eta \notin h^*JC.
\end{cases}
$$
\end{lemma}

\begin{proof}
Since $h$ is ramified, the induced homomorphism $h^*:JC\to JC'$ is injective. As $f^*=k^*\circ h^*$, we have
$$
h^*(\ker f^*)=\ker(k^*)\cap h^*JC.
$$
Since $k$ is an \'etale cyclic cover of prime degree $p$ defined by $\eta$, one has $\ker(k^*)=\langle\eta\rangle\simeq\mathbb Z_p$, and therefore $\ker f^*$ is trivial if and only if $\eta\notin h^*JC$ and thus has order $p$ if $\eta\in h^*JC$. Let $d=pq$ be the degree of the cover $f$. According to Proposition \ref{prop:kernel_cover},
$$
|K(\widetilde{\Theta}_{f^*JC})|=\frac{|JC[d]|}{|\ker f^*|^2}=\frac{d^2}{|\ker f^*|^2}.
$$
Since $\dim f^*JC=1$, we have that the induced polarization has type $\left(\frac{d}{|\ker f^*|}\right)$, and the claim follows.
\end{proof}

\begin{lemma}\label{lemma:kernel_mu_tilde}
Let $f:\widetilde{C}\to C$ be a cover of smooth projective curves which factors as $f=h\circ k$, where $k:\widetilde{C}\to C'$ is an \'etale cyclic cover of degree $p$ defined by $h^*(\xi)=\eta\in JC'[p]$ where $\xi\in JC[p]$. Let
$$
\mu: h^*JC\times P(h)\longrightarrow JC'\quad \text{and}\quad \tilde\mu: k^*h^*JC\times k^*P(h)\longrightarrow k^*JC'
$$
be the corresponding addition maps. Then
$$ \ker(\tilde\mu) \simeq\ker(\mu)/ (\ker(k^*)\cap \ker \mu).$$
Moreover, if $h$ is a ramified cover of degree $p$, then $\eta\in P(h)$ and
$$ \ker(\tilde\mu) \simeq\ker(\mu)/ \ker(k^*).$$
\end{lemma}

\begin{proof}
Consider the commutative diagram
\begin{equation}\label{eq:diagram-mu}
\begin{tikzcd}[row sep=large, column sep=large]
h^*JC\times P(h) \ar[r,"\mu"] \ar[d,"k^*\times k^*"'] 
& JC' \ar[d,"k^*"] \\
k^*h^*JC\times k^*P(h) \ar[r,"\tilde\mu"] 
& k^*JC'.
\end{tikzcd}
\end{equation}

Restricting $k^*\times k^*$ to $\ker\mu$ we obtain a homomorphism
$$
\alpha:=(k^*\times k^*)|_{\ker\mu}:\ker\mu\longrightarrow\ker\tilde\mu.
$$

The map $\alpha$ is well defined by commutativity of \eqref{eq:diagram-mu}. We claim that $\alpha$ is surjective. Indeed, let $(z,-z) \in\ker\tilde\mu$. Since $k^*h^*JC\times k^*P(h)$ is the image of $k^*\times k^*$, there exist $x\in h^*JC$ and $y\in P(h)$ such that $z=k^*x=k^*y$. Hence $k^*(x-y)=0$, and since $\ker k^*=\langle\eta\rangle\simeq\mathbb Z_p$, we have  $y=x-t\eta$ for some $t\in\{0,\dots,p-1\}$. Note that $x,\eta\in h^*JC$ implies $y\in h^*JC\cap P(h)$, thus $(y,-y)\in\ker\mu$  and
$$
\alpha(y,-y)=(k^*y,-k^*y)=(z,-z),
$$
which proves surjectivity.

Then we have that
$$
\ker\alpha=\{(x,-x)\in\ker\mu:\ k^*x=0\}
\simeq \ker k^*\cap \ker \mu
$$
and the result follows by the first isomorphism theorem. Now suppose that $h$ is ramified of degree $p$, then $h^*$ is injective and \cite[Proposition 3.2.4]{LR2022} implies that $\ker\Nm_h$ is connected. Note that for $\eta\in JC'[p]$ 
$$\Nm_h(\eta)=\Nm_h(h^*\xi)=\deg(h)\cdot \xi=0$$
since $\xi$ is a $p$-torsion point of $JC$. Then $\eta\in P(h)$. Since also $\eta=h^*\xi\in h^*JC$, we have
$$
\ker k^*=\langle\eta\rangle\subset h^*JC\cap P(h)\simeq\ker\mu.
$$
Thus $\ker k^*\cap\ker\mu=\ker k^*$, and the result follows.
\end{proof}

Now we want to determine the type of the induced polarization on $k^*P(h)$. We have the following result.

\begin{theorem}\label{lemma:PTlemma}
Let $f:\widetilde{C}\to C$ be a cover of smooth projective curves over an elliptic curve $C$ which factors as $f=h\circ k$, where $k:\widetilde{C}\to C'$ is an \'etale cyclic cover of degree $p$ defined by $\eta\in h^*JC[p]\subset JC'$, and $h:C'\to C$ is a ramified cover of degree $q$ with ramification degree $r(h)\ge 4$, and $p$ and $q$ odd primes. Then $P:=k^*P(h)$ is a Prym--Tyurin variety if and only if $p=q$ and $\exp K(\widetilde\Theta_{P})=p$.
\end{theorem}

\begin{proof}
Assume that $P$ is a Prym--Tyurin variety of exponent $\varepsilon\geq 2$. The hypothesis on the ramification degree and the Riemann--Hurwitz formula imply that $\dim P=d\ge 2$, then $\widetilde\Theta_P$ has type $(\varepsilon,\dots,\varepsilon)$ with $d$ entries equal to $\varepsilon$. According to Corollary \ref{cor:transitivity}, we can determine the type of $P$ studying its induced polarization as a subvariety of $k^*JC'$. Consider the isogenies
$$\tilde{\mu}:k^*h^*JC\times k^*P(h)\to k^*JC', \quad \mu:h^*JC\times P(h)\to JC'.$$
Since $k$ is an \'etale cyclic cover of degree $p$ and $h$ is ramified of prime degree, Proposition \ref{prop:type_prym} implies that the type of $\widetilde{\Theta}_{k^*JC'}$ is $(1,p,\dots,p)$ with $d$ entries equal to $p$, $\Theta'_{h^*JC}$ has type $(q)$ and $\Theta'_{P(h)}$ has type $(1,\dots,1,q)$. Moreover, since $\eta\in h^*JC[p]\subset JC'$, Lemma \ref{lemma_type} shows that the induced polarization on $k^*h^*JC=f^*JC$ has type $(q)$. Applying Lemma \ref{lemma:kernel} for $\tilde{\mu}$ we obtain 
\begin{equation}\label{eq01}
p^{2d}\,|\ker\tilde{\mu}|^2 = q^2\,\varepsilon^{2d}.
\end{equation}
and applying Lemma \ref{lemma:kernel} for $\mu$ we obtain
$$
|\ker\mu| = q^2.
$$

Since $h$ is ramified of degree $q$, then $h^*$ is injective and \cite[Proposition 3.2.4]{LR2022} implies that $\ker\Nm_h$ is connected. Note that, since $\eta=h^*\xi$ with $\xi\in JC[p]$, 
$$\Nm_h(\eta)=\Nm_h(h^*\xi)=\deg(h)\cdot \xi=q\xi.$$
Suppose that $p\neq q$. Since $\xi$ is a $p$-torsion point of $JC$, $\xi\neq 0$ implies that $q\xi\neq 0$, thus $\eta\notin P(h)$, $\ker k^*\cap \ker \mu=0$ and then according to Lemma \ref{lemma:kernel_mu_tilde} $\ker \tilde\mu\simeq \ker \mu$. Therefore \eqref{eq01} implies that $\varepsilon^{2d}=p^{2d}q^2$, which is impossible since $d\ge 2$. Hence $p=q$.

Then $\Nm_h(\eta)=p\xi=0$, and since $\ker\Nm_h$ is connected, we conclude that $\eta\in P(h)$. Then Lemma \ref{lemma:kernel_mu_tilde} implies that $|\ker\tilde{\mu}|=p$, and thus $\ker \tilde{\mu}\simeq \mathbb{Z}_p$. Then \eqref{eq01} implies that $p=\varepsilon$.

For the converse, suppose that $p=q$ and that $\exp K(\widetilde\Theta_P)=p$. Then the type of the induced polarization $\widetilde{\Theta}_P$ is $(1,\dots,1,p,\dots,p)$, with $r$ entries equal to $1$ and $d-r$ entries equal to $p$. By Lemma \ref{lemma:kernel} applied to $\tilde{\mu}$, we obtain that
$$
|\ker\tilde{\mu}| = p^{1-r}.
$$
As above, since $p=q$, we have $\eta\in P(h)$ and Lemma \ref{lemma:kernel_mu_tilde} gives $\ker\tilde\mu\simeq \mathbb Z_p$. We conclude that $r=0$ which implies that $\widetilde{\Theta}_P$ is $p$ times a principal polarization.
\end{proof}

Now let $f:\widetilde C\to C$ be an \'etale cover of degree $d$. Since $\ker f^*\subset JC[d]=K(d\Theta_C)$, the Weil pairing $e^{d\Theta_C}$ gives a natural notion of isotropy for the cover. Following \cite{BoOr20}, we say that $f$ is \emph{isotropic} if $\ker f^*$ is isotropic with respect to $e^{d\Theta_C}$; otherwise, $f$ is called \emph{non-isotropic}.

We now consider covers $f:\widetilde C\to C$ which factorize as $f=h\circ k$, where $k:\widetilde C\to C'$ and $h:C'\to C$ are cyclic \'etale covers of degrees $p$ and $q$, respectively, with $p,q$ odd primes. If $p\neq q$, then $f$ is cyclic of degree $pq$, and $\ker f^*\subset JC[pq]$ is a cyclic subgroup of order $pq$. Hence $f$ is automatically isotropic.

Thus the distinction between isotropic and non-isotropic covers becomes relevant for the elementary abelian case $p=q$, namely for \'etale Galois covers with group $\mathbb Z_p\times\mathbb Z_p$. In this case $\ker f^*\subset JC[p]$ is isomorphic to $\mathbb Z_p\times\mathbb Z_p$, and the two possibilities give different polarization types for $P(f)$, as described below.

The following result is a direct generalization of Lemma 2.5 in
\cite{BoSh25}.

\begin{lemma}\label{lem:type_p2}
Let $f:\widetilde{C}\to C$ be an \'etale Galois cover with Galois group $G=\mathbb{Z}_p\times \mathbb{Z}_p$, where $p$ is an odd prime, and let $C$ be a curve of genus $g\ge 2$. Let $P(f)$ be the Prym variety of $f$, endowed with the polarization induced by $\widetilde{\Theta}$. Then:

\begin{enumerate}
\item if the covering $f$ is isotropic, the induced polarization on $P(f)$ has
type
$$
(1,\ldots,1,\; p,p,\; p^2,\ldots,p^2),
$$
with $g-2$ entries equal to $p^2$;

\item if the covering $f$ is non--isotropic, the induced polarization on $P(f)$
has type
$$
(1,\ldots,1,\; p^2,\ldots,p^2),
$$
with $g-1$ entries equal to $p^2$.
\end{enumerate}
\end{lemma}

\begin{proof}
Let $e^{p^2\Theta_C}$ denote the Weil pairing on $JC[p^2]$. As in \cite[Proposition 3.2.9]{LR2022}, one has $K(\Theta_{P(f)}) \simeq (\ker f^{*})^{\perp_{p^2\Theta_C}}/\ker f^{*}$.

Choose symplectic generators $e_1,f_1,\ldots,e_g,f_g$ of $JC[p^2]$ such that $e^{p^2\Theta_C}(e_i,f_i)=\exp(2\pi i/p^2)$ for $i=1,\ldots,g$.

In the isotropic case, up to a symplectic change of basis we may assume $\ker f^{*}=\langle p e_1,\; p e_2\rangle$. Then $(\ker f^{*})^{\perp_{p^2\Theta_C}}=\langle e_1,\; p f_1,\; e_2,\; p f_2,\; e_3,\; f_3,\;\ldots,\; e_g,\; f_g\rangle$, and hence
$$K(\Theta_{P(f)})\simeq (\ker f^{*})^{\perp_{p^2\Theta_C}}/\ker f^{*}
\simeq \mathbb{Z}_p^2\times \mathbb{Z}_p^2 \times (\mathbb{Z}_{p^2})^{2(g-2)},
$$
which corresponds to polarization type $(1,\ldots,1,p,p,p^2,\dots, p^2)$ with $g-2$ entries equal to $p^2$. 

In the non-isotropic case, similarly we may assume $\ker f^{*}=\langle p e_1,\; p f_1\rangle$.  Then
$$(\ker f^{*})^{\perp_{p^2\Theta_C}}=\langle p e_1,\; p f_1,\; e_2,\; f_2,\;\ldots,\; e_g,\; f_g\rangle,$$
and thus
$$K(\Theta_{P(f)})\simeq(\ker f^{*})^{\perp_{p^2\Theta_C}}/\ker f^{*}
\simeq (\mathbb{Z}_{p^2})^{2(g-1)},
$$
which corresponds to polarization type $(1,\ldots,1,p^2,\dots,p^2)$ with $g-1$ entries equal to $p^2$.
\end{proof}

\begin{theorem}\label{lemma:PTlemmagenus2}
Let $f:\widetilde{C}\to C$ be an isotropic cover of smooth projective curves, where $C$ is a curve of genus $2$ and suppose that it factorizes as $f=h\circ k$, with $k:\widetilde{C}\to C'$ and $h:C'\to C$ \'etale covers of degree $p$ with $p$ an odd prime. Then $P=k^*P(h)$ is a Prym--Tyurin variety of exponent $p$.
\end{theorem}

\begin{proof}
By \eqref{eq:RH_formula}, the genus of $\widetilde C$ is $p^{2}+1$ and the genus of $C'$ is $p+1$. Hence $\dim P = g(C')-g(C)=p-1.$

According to \cite[Lemma 2.10]{BoSh25}, in the isotropic case $\eta\in P(h)$. Set $u:=k^*_{|P(h)}:P(h)\to P$. Then $\ker u=\langle\eta\rangle$. Since $\eta\in h^*JC\cap P(h)=K(\Theta'_{P(h)})$ by Proposition \ref{prop:kernel_cover}, and $K(\Theta'_{P(h)})\simeq(\mathbb Z/p)^2$, the subgroup $\langle\eta\rangle$ is maximal isotropic with respect to $e^{\Theta'_{P(h)}}$. Thus $\Theta'_{P(h)}$ descends to a polarization $M$ on $P$. Moreover, \cite[Proposition 3.2.1(a)]{LR2022} gives $u^*\widetilde\Theta_P=(\Theta'_{P(h)})^{\otimes p}$, hence $\widetilde\Theta_P=M^{\otimes p}$. By Proposition \ref{prop:exponent_descended_polarization}, $\exp K(\widetilde\Theta_P)=p$. Thus $\widetilde{\Theta}_P$ has type
$$
(1,\ldots,1,p,\ldots,p),
$$
with $r$ entries equal to $1$ and $p-1-r$ entries equal to $p$. Therefore
$$
|K(\widetilde\Theta_{P})| = p^{2(p-1-r)}.
$$

Moreover, by Lemma \ref{lem:type_p2}, the induced polarization on $k^{*}h^{*}JC$ has type $(p,p)$. Consider the isogeny
$$
\tilde{\mu}: k^{*}h^{*}JC \times P \longrightarrow k^*JC'.
$$
Applying Lemma \ref{lemma:kernel} to $\tilde{\mu}$, we obtain that
$$
|\ker\tilde{\mu}| = p^{1-r}.
$$

Lemma \ref{lemma:kernel_mu_tilde} implies that $\ker \tilde{\mu}\simeq \mathbb{Z}_p$, so we conclude that $r=0$ and then $P$ is a Prym--Tyurin variety of exponent $p$.
\end{proof}

The following result shows that in other covers the subvariety $P$ is not a Prym--Tyurin variety for $\widetilde C$.

\begin{proposition}\label{prop:exp_p2_general_genus}
Let $p$ be an odd prime and let $f=h\circ k:\widetilde C\to C$ be a cover of smooth projective curves, where $k:\widetilde C\to C'$ is an \'etale cyclic cover of degree $p$ and $h:C'\to C$ is a cyclic cover of degree $p$. Let $g$ be the genus of $C$ and set $P:=k^*P(h)\subset J\widetilde C$. If $g\ge3$, then the polarization induced by $\widetilde\Theta$ on $P$ has exponent $p^2$, and it is not equal to $p^2$ times a principal polarization.
\end{proposition}

\begin{proof}
By Lemma \ref{lemma:exponents}, the elementary divisors of $\widetilde\Theta_P$ belong to $\{1,p,p^2\}$ and $\exp K(\widetilde\Theta_P)$ divides $p^2$. Let $u$, $v$ and $w$ denote respectively the number of times that $1$, $p$ and $p^2$ occur among the elementary divisors of $\widetilde\Theta_P$ and denote by $g'$ the genus of $C'$. Then
$$
u+v+w=\dim P=g'-g
\quad\text{and}\quad
|K(\widetilde\Theta_P)|=p^{2v+4w}.
$$

Applying Lemma \ref{lemma:kernel} to the addition map $\widetilde\mu:f^*JC\times P\to k^*JC'$ and using the computation of the relevant kernels obtained above, we get
$$
|K(\widetilde\Theta_P)|=
\begin{cases}
p^{2g'-4}, & \text{if $h$ is \'etale or $|\ker f^*|=1$,}\\
p^{2g'-2}, & \text{if $h$ is ramified and $|\ker f^*|=p$.}
\end{cases}
$$
Hence
\begin{equation}\label{eq:dimP}
v+2w=
\begin{cases}
g'-2,& \text{if $h$ is \'etale or $|\ker f^*|=1$,}\\
g'-1,& \text{if $h$ is ramified and $|\ker f^*|=p$.}
\end{cases}
\end{equation}

Since $\dim P=g'-g$, if $g\ge3$ then $v+2w>\dim P$. If $w=0$, then $v\le u+v=\dim P$, a contradiction. Thus $w\ge1$, and therefore $\exp K(\widetilde\Theta_P)=p^2$.

Finally, if $\widetilde\Theta_P$ were equal to $p^2$ times a principal polarization, then $u=v=0$, so that $v+2w=2\dim P$, which is incompatible with \eqref{eq:dimP} when $g\ge3$.
\end{proof}

\section{Prym--Tyurin varieties coming from Galois covers}\label{section4}
As a consequence of Theorem \ref{lemma:PTlemma}, the search for Prym--Tyurin varieties arising as a Prym of an intermediate covering of the Jacobian of a curve endowed with a group action by $G$ naturally leads to the study of Galois covers of elliptic curves whose Galois group contains a subgroup isomorphic to $\mathbb{Z}_p \times \mathbb{Z}_p$. In this section, we investigate in more detail such Galois covers. 
\subsection{Prym--Tyurin varieties via  $\mathbb{Z}_p\times \mathbb{Z}_p$ action}
Let $$G_p=\left\langle a,b\,\big|\,a^p=b^p=1,ab=b a\right\rangle\simeq \mathbb{Z}_p\times \mathbb{Z}_p$$ with $p$ an odd prime number, and let $f:\widetilde{C}\to C$ be a Galois covering with $G_p$ as its Galois group. Let $C_g$ denote the quotient curve $\widetilde{C}/\left\langle g\right\rangle$, with $g\in G_p$. Consider $X=\{a,b,ab,a^2b,\dots,a^{p-1}b\}$, a set of  generators for the non-trivial cyclic subgroups of $G_p$. The subcoverings of $\widetilde{C}$ are as follows:
 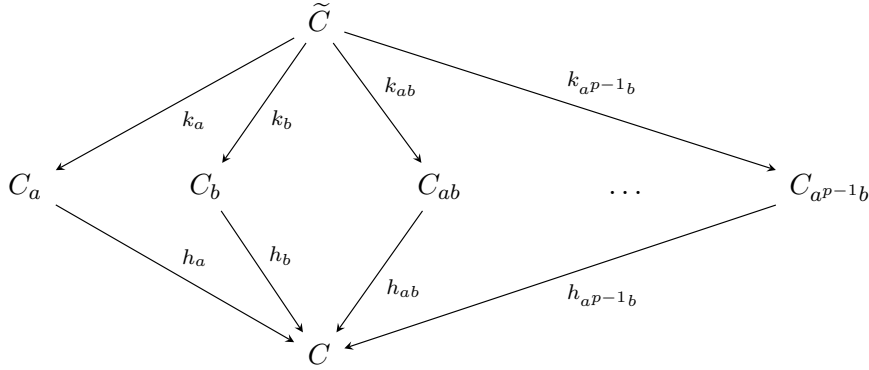
\begin{figure}[hbt!]
    \centering 
    \begin{tikzcd}[row sep=1.6cm, column sep=0.8cm]
    & & & \widetilde{C}\ar[dlll,"k_a"]\ar[dl,"k_b"]\ar[dr,"k_{ab}"]\ar[drrrrr,"k_{a^{p-1}b}"]& & & & \\
    C_a\ar[drrr,"h_a"]& &C_b\ar[dr,"h_b"] & & C_{ab}\ar[dl,"h_{ab}"] & &\dots & &C_{a^{p-1}b}\ar[dlllll,"h_{a^{p-1}b}"] \\
    & & & C& & & & &\\
    \end{tikzcd}
    \caption{Subcoverings of $\widetilde{C}$ by the action of $G_p$.}
    \label{fig:subcoverings_C3xC3}
\end{figure} 

We have the following result.
\begin{proposition}\label{prop:iso_decomposition}
The Jacobian $J\widetilde C$ of $\widetilde{C}$ decomposes as
$$J\widetilde C\sim JC\times \prod_{g\in X}P(h_g).$$
\end{proposition}
\begin{proof}
The \, group \,$\mathbb{Z}_p$ \, has \, $p$ \, irreducible \, complex \, representations \, of \, degree \, $1$, \, denoted \, by \, $1,\omega,\dots,\omega^{p-1}$, where each of them maps the generator of the group to $\omega^j$ where $j=0,\dots,p-1$ and $\omega=e^{2\pi i/p}$. We have that the irreducible complex representations of $\mathbb{Z}_p\times \mathbb{Z}_p$ are given by $\{\omega^i\otimes \omega^j:0\leq i,j\leq p-1\}$.

Following \cite[Section 2.8]{LR2022}, we obtain the irreducible rational representations of $\Z_p\times \mathbb{Z}_p$:
$$ W_0=1\otimes 1,\quad W_{1\otimes \omega}=\bigoplus_{j=1}^{p-1}1\otimes \omega^j,\quad  W_{\omega\otimes \omega^i}=\bigoplus_{j=1}^{p-1}\omega^j\otimes \omega^{ij}:1\leq i\leq p$$

Thus, we have $p+2$ rational conjugacy classes in $G_p$. Actually, the rational conjugacy classes are the trivial class and the class of each element in $X$. Denote by $[g]$ the rational class of $g\in G_p$, then we have that the rational characters are given by:
\begin{table}[h!]
\centering
\small
\caption{Rational character table of the group $G_p$}
\label{table:C3xC3CharacterTable}
\begin{tabular}{>{$}c<{$}>{$}c<{$}>{$}c<{$}>{$}c<{$}>{$}c<{$}>{$}c<{$}>{$}c<{$}>{$}c<{$}>{$}c<{$}}\hline
G & 1 &  [a] &  [b] & [ab]  & \dots  &  [a^{-j}b] & \dots  & [a^{p-1}b]   \\\hline
W_0 & 1               & 1 & 1 & 1 &  & 1 & 1 & 1   \\
W_{1\otimes \omega} & p-1          & p-1 & -1 & -1 &  & -1 &-1 & -1     \\
W_{\omega\otimes 1} & p-1  & -1 & p-1 & -1 &   & -1 & -1 &-1   \\
\vdots &           &  &  &  & &  &  &   \\
W_{\omega\otimes \omega^j} & p-1  & -1 & -1 & -1 &  & p-1 &  & -1   \\
\vdots &  &  &  &  &  & &  &    \\
W_{\omega\otimes \omega^{p-1}} &  p-1  & -1 & -1 & -1 &  & -1 &   & p-1   \\\hline
\end{tabular}
\end{table}

On the other hand, it is not difficult to see that for each $g\in X$ the character $\chi_g$ of the representation induced on $G_p$ from the trivial representation of the subgroup generated by $g$ is given by 
$$\chi_g(h)=\begin{cases}
   p & \mbox{if } h=1\\
   p & \mbox{for each } h\mbox{{ in the rational class of }g,}\\
   0 & \mbox{otherwise.}
\end{cases}$$

Therefore we obtain that the representations induced on $G_p$ from the trivial representation of each subgroup of $G_p$ are given by
$$\rho_{a}=W_0\oplus W_{1\otimes \omega},\quad\rho_{a^{-j}b}=W_0\oplus W_{\omega\otimes \omega^j},\quad 1\leq j\leq p $$

and we have the following diagram

 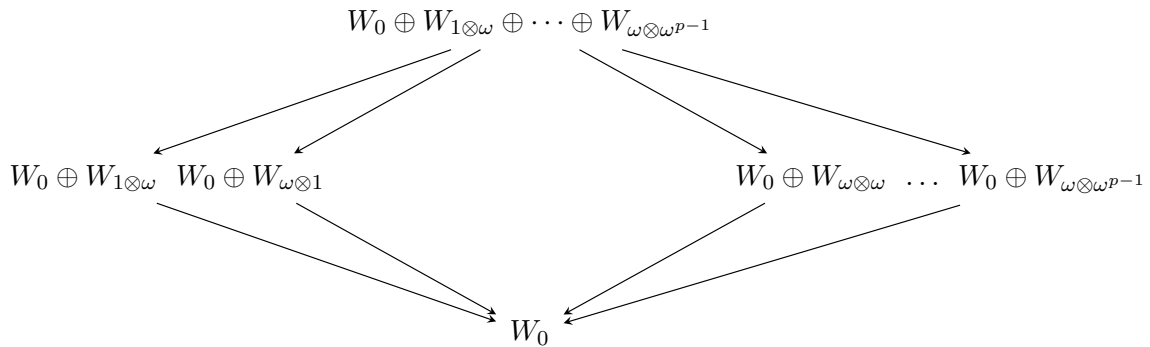
\begin{figure}[H]
    \centering 
    \begin{tikzcd}[row sep=huge, column sep=-0.05cm]
    & & & W_0\oplus W_{1\otimes \omega}\oplus\dots \oplus W_{\omega\otimes \omega^{p-1}}\ar[dlll]\ar[dl]\ar[dr]\ar[drrrrr]& & & & \\
    W_0\oplus W_{1\otimes \omega}\ar[drrr]& &W_0\oplus W_{\omega\otimes 1}\ar[dr] & & W_0\oplus W_{\omega\otimes \omega} \ar[dl] & &\dots & &W_0\oplus W_{\omega\otimes \omega^{p-1}}\ar[dlllll] \\
    & & & W_0& & & & &\\
    \end{tikzcd}
    \caption{Induced representation by trivial representations of subgroups of $G_p$.}
    \label{fig:induced_C3xC3}
\end{figure} 

According to \cite[Corollary 3.5.10]{LR2022}, we have that 
$$P(h_g)\sim A_{W_g}.$$

This implies the assertion.
\end{proof}

\begin{remark}\label{rmk:vector}
Let $c_1,\dots,c_m\in\{1,\dots,p-1\}$ be such that
$$\sum_{j=1}^m c_j\equiv0\pmod p.$$
Then every tuple
$$v=(\alpha,\beta;a^{c_1},\dots,a^{c_m}),\qquad \alpha,\beta\in G_p,$$
such that $\langle a,\alpha,\beta\rangle=G_p$ is a generating vector of type $s=(1;p,\dots,p)$ with $m$ entries equal to $p$. Indeed, the product relation follows from the congruence, since $G_p$ is abelian, and the generating condition holds by assumption. Such vectors exist; for instance, one may take $\alpha=1$ and $\beta=b$. Therefore Theorem \ref{theo:Riemann:existence} gives a curve $\widetilde C$ of genus $\frac{m}{2}p(p-1)+1$ endowed with the desired $G_p$--action. Conversely, every generating vector of this type whose local monodromies are non--trivial powers of $a$ is of this form. Thus these are precisely the generating vectors producing the $\mathbb Z_p\times\mathbb Z_p$--actions required in the construction.
\end{remark}

We have the following result.
\begin{theorem}\label{Theo1}
For each odd prime $p$, each integer $m\ge 2$ and each generating vector $v$ as in Remark \ref{rmk:vector}, there exists an $m$--dimensional family of curves $\widetilde{C}$ of genus $\frac{m}{2}p(p-1)+1$ endowed with an action of $G_p$ with signature $s$ and generating vector $v$, such that the abelian subvarieties $k_g^*P(h_g)\subset J\widetilde{C}$, for $g\in X\setminus\{a\}$, are Prym--Tyurin varieties of exponent $p$ and dimension $\frac{m}{2}(p-1)$.
\end{theorem}

\begin{proof}
Let $\widetilde C$ be a curve with the $G_p$-action described in the statement, and set $E:=\widetilde C/G_p$. By Riemann--Hurwitz, $g(\widetilde C)=\frac{m}{2}p(p-1)+1$. Using \cite[Corollary 3.5.17]{LR2022}, the dimensions of the group algebra components associated with this action are
$$
\dim B_{W_0}=\dim JE=1,\quad \dim B_{W_{1\otimes\omega}}=\dim P(h_a)=0,
$$
and
$$
M:=\dim B_{W_{\omega\otimes\omega^j}}=\dim P(h_{a^{-j}b})=\frac{m}{2}(p-1),\qquad 1\le j\le p.
$$
By Proposition \ref{prop:iso_decomposition}, for every $g\in X\setminus\{a\}$, the Prym variety $P(h_g)$ is isogenous to the component $B_{W_g}$. Hence $\dim P(h_g)=\frac{m}{2}(p-1)$. Set
$$
P_g:=k_g^*P(h_g)\subset J\widetilde C
$$
and denote by
$$
u_g:=k_g^*|_{P(h_g)}:P(h_g)\longrightarrow P_g
$$
the induced isogeny.

We now verify the hypotheses needed to apply Theorem \ref{lemma:PTlemma}. The ramification of the $G_p$-cover is generated by elements of the subgroup $\langle a\rangle$. Therefore $h_a:C_a\to E$ is an étale cyclic cover of degree $p$. On the other hand, if $g\in X\setminus\{a\}$, then $k_g:\widetilde C\to C_g$ is étale of degree $p$. Moreover, the image of $\langle a\rangle$ in $G_p/\langle g\rangle$ is non-trivial, so $h_g:C_g\to E$ is a ramified cyclic cover of degree $p$. Its ramification degree is at least $4$, because $m\ge2$ and $p$ is odd.

Let $\eta_g\in JC_g[p]$ be the point defining the étale cyclic cover $k_g:\widetilde C\to C_g$. Since $f:\widetilde C\to E$ factors through the \'etale cover $h_a:C_a\to E$ and also as $f=h_g\circ k_g$, the argument in the proof of Lemma \ref{lemma_type} gives
$$
\ker k_g^*\cap h_g^*JE\neq 0.
$$
Hence, as $\ker k_g^*=\langle\eta_g\rangle$, we have $\eta_g\in h_g^*JE[p]$. Therefore $\eta_g=h_g^*\xi$ for some $\xi\in JE[p]$. Since $h_g$ is ramified of degree $p$, $h_g^*:JE\to JC_g$ is injective and $\ker\Nm_{h_g}$ is connected. Moreover $\Nm_{h_g}(\eta_g)=\Nm_{h_g}(h_g^*\xi)=p\xi=0$, so $\eta_g\in P(h_g)$. In particular, $\ker u_g=\langle\eta_g \rangle$.

Let $\widetilde\Theta_{P_g}$ be the polarization induced by $\widetilde\Theta$ on $P_g$. Since $h_g:C_g\to E$ is a ramified cyclic cover of degree $p$ over an elliptic curve, the induced polarization on $P(h_g)$ has type $(1,\ldots,1,p)$, and hence we have that $K(\Theta'_{P(h_g)})\simeq(\mathbb Z/p)^2$. Moreover, $\eta_g\in h_g^*JE\cap P(h_g)=K(\Theta'_{P(h_g)})$ by Proposition \ref{prop:kernel_cover}, so $\langle\eta_g\rangle$ is maximal isotropic with respect to $e^{\Theta'_{P(h_g)}}$, since $\langle\eta\rangle$ has index $p$. Therefore $\Theta'_{P(h_g)}$ descends to a polarization $M_g$ on $P_g$. Since \cite[Proposition 3.2.1(a)]{LR2022} gives $u_g^*\widetilde\Theta_{P_g}=(\Theta'_{P(h_g)})^{\otimes p}$, we have $\widetilde\Theta_{P_g}=M_g^{\otimes p}$. Proposition \ref{prop:exponent_descended_polarization} gives $\exp K(\widetilde\Theta_{P_g})=p$.

Since $k_g$ is étale cyclic of degree $p$ defined by $\eta_g\in h_g^*JE[p]\subset JC_g[p]$,  $h_g$ is ramified of degree $p$ with ramification degree at least $4$ and the exponent of $P_g$ is $p$, Theorem \ref{lemma:PTlemma} applies. Hence $P_g=k_g^*P(h_g)$ is a Prym--Tyurin variety of exponent $p$.

Finally, the dimension of the family of curves with this action in the moduli space $\mathcal M_{pM+1}$ is determined by the number of branch points of the action and by the fact that the quotient curve $\widetilde{C}/G_p$ has genus $1$.
\end{proof}

\begin{corollary}
Let
$$\mu_1: f^*JC\times \prod_{g\in X} k_g^*P(h_g)\to J\widetilde C,\quad \mu_2: \prod_{g\in X} k_g^*P(h_g)\to P(f)$$
be the addition maps. Then 
$$\deg \mu_1=p^{\widetilde{g}}\quad \mbox{and}\quad \deg \mu_2=p^{\widetilde{g}-2}.$$
\end{corollary}

\begin{proof}
The degree of $\mu_1$ is a direct consequence of applying Lemma \ref{lemma:kernel_mu}. The degree of $\mu_2$ can be calculated using Lemma \ref{lemma:kernel_mu} and the fact that the type of the induced polarization on $P(f)$ is $(1,\dots,1,p)$, since it is a consequence of Lemma \ref{lemma_type}, Theorem \ref{Theo1} and  \cite[Proposition 2.6.3]{LR2022}.
\end{proof}

Theorem \ref{Theo1} admits a completely \'etale analog over a base curve of genus $2$, provided the covering is isotropic.

\begin{theorem}\label{Theo1_genus2}
Let $p$ be an odd prime. There exists a $3$--dimensional family of curves $\widetilde C$ endowed with an \emph{isotropic} \'etale action of $G_p\simeq \mathbb Z_p\times \mathbb Z_p$ such that, for every $g\in X$, the abelian subvariety
$$
P_g:=k_g^*P(h_g)\subset J\widetilde C
$$
is a Prym--Tyurin variety of exponent $p$ and dimension $p-1$.
\end{theorem}

\begin{proof}
Since all coverings involved are \'etale and $g(C)=2$, the hypotheses of Theorem \ref{lemma:PTlemmagenus2} are satisfied for every subgroup $\langle g\rangle\subset G_p$, and therefore each $P_g=k_g^*P(h_g)$ is a Prym--Tyurin variety of exponent $p$ of dimension $p-1$.

Finally, the only moduli parameter in this construction is the base curve $C$ of genus $2$, whose moduli space has dimension $3$, and the choice of an \'etale isotropic $G_p$--cover over a fixed $C$ is discrete. Hence the family has dimension $3$.
\end{proof}

As in the elliptic case, the degrees of the addition maps can be computed directly from Proposition \ref{prop:iso_decomposition}, the isotypical decomposition of $J\widetilde C$, and Lemma \ref{lemma:kernel_mu}, together with the description of the induced polarizations in the isotropic case.

\begin{corollary}
In the situation of Theorem \ref{Theo1_genus2}, let
$$
\mu_1: f^*JC\times \prod_{g\in X} k_g^*P(h_g)\longrightarrow J\widetilde C,
\quad
\mu_2: \prod_{g\in X} k_g^*P(h_g)\longrightarrow P(f)
$$
be the addition maps. Then
$$
\deg \mu_1 = p^{\,g(\widetilde C)}
\quad\text{and}\quad
\deg \mu_2 = p^{\,g(\widetilde C)-4}.
$$
\end{corollary}

\section{Maps associated with the Prym--Tyurin construction}\label{sec5}

\subsection{The Prym--Tyurin map}\label{section5}

Fix an odd prime number $p$. Theorem \ref{lemma:PTlemma} implies that, given a cyclic covering $h:C\to E$ of degree $p$ of an elliptic curve and a non-trivial $p$--torsion point $\xi\in E[p]$, with $\eta=h^*(\xi)\in h^*JE\cap P(h)\subset JC[p]$ determining an \'etale covering $k:\widetilde C\to C$ of degree $p$, the associated abelian subvariety
$$
P:=k^*P(h)\subset J\widetilde C
$$
is then a Prym--Tyurin variety of exponent $p$.

Fix a generator $a$ of the Galois group of $h$. The cyclic covering is described by a generating vector
$$
v=(a^r,a^s;a^{c_1},\dots,a^{c_m}),
\qquad r,s\in\{1,\dots,p\},
$$
where $c_j\in\{1,\dots,p-1\}$ and $\sum_jc_j\equiv0\pmod p$. Recall from subsection \ref{section_coverings} that the corresponding algebraic data are given by a triple $(E,B,L)$, where
$$
B=\sum_{j=1}^m c_jx_j
$$
is an effective divisor and $L\in\Pic(E)$ satisfies
$$
L^{\otimes p}\simeq \mathcal O_E(B).
$$
The local monodromies $a^{c_1},\dots,a^{c_m}$ determine the weights of $B$, while the full generating vector determines the choice of $L$ with $L^{\otimes p}\simeq\mathcal O_E(B)$. The reduced branch divisor is $D=x_1+\cdots+x_m$.

Fix one such vector $v$, and let $\mathcal R(p,m)$ be the corresponding moduli space of triples $(E,B,L)$. Notice that the choice of $v$ is discrete; allowing it to vary therefore gives a family of such moduli spaces and of the maps defined below. Let
$$
\mathcal R(p,m)^{\xi}:=\{(E,B,L,\xi)\mid (E,B,L)\in\mathcal R(p,m),\ \xi\in E[p]\setminus \{0\}\}
$$
be the finite covering obtained by marking a $p$--torsion point. The cyclic cover determined by $(E,B,L)$ has a Prym variety $P(h)$ of dimension $\frac{m}{2}(p-1)$, and the Prym--Tyurin variety $P=k^*P(h)$ constructed from $(E,B,L,\xi)$ has the same dimension. We thus obtain two natural morphisms: the Prym map
$$
\mathcal P_{p,m}:\mathcal R(p,m)\longrightarrow \mathcal A_{\frac{m}{2}(p-1)}^{\delta},
\quad
(E,B,L)\longmapsto P(h),
$$
where $\delta=(1,\dots,1,p)$, and the Prym--Tyurin map
$$
\mathcal{PT}_{p,m}:\mathcal R(p,m)^{\xi}\longrightarrow \mathcal A_{\frac{m}{2}(p-1)},
\quad
(E,B,L,\xi)\longmapsto (P,\Xi).
$$
Since $\mathcal R(p,m)^{\xi}\to \mathcal R(p,m)$ is finite, the map $\mathcal{PT}_{p,m}$ is generically finite onto its image if and only if $\mathcal P_{p,m}$ is generically finite onto its image. In what follows we compute the codifferential of $\mathcal P_{p,m}$ at a general point and prove its surjectivity, following \cite{LO2011}.

The projection $\mathcal R(p,m)\to\mathcal M_{1,m}$ is finite, because once $(E,x_1,\dots,x_m)$ is fixed, there are only finitely many choices of $L$ with $L^{\otimes p}\simeq\mathcal O_E(B)$. In particular, at a general point $(E,B,L)$ the tangent space of $\mathcal R(p,m)$ identifies with the tangent space of the pointed elliptic curve $(E,x_1,\dots,x_m)$, namely
$$
T_{(E,B,L)}\mathcal R(p,m)\simeq H^1(E,T_E(-D)).
$$
By Serre duality $T_{(E,B,L)}^\vee\mathcal R(p,m)\simeq H^0(E,\omega_E\otimes T_E^\vee(D))$. Since $E$ is elliptic, both $\omega_E$ and $T_E$ are trivial, and therefore
$$
T_{(E,B,L)}^\vee\mathcal R(p,m)\simeq H^0(E,\mathcal O_E(D)).
$$

To compute the codifferential, we describe the $\mathbb Z_p$--eigenspace decomposition of $H^0(C,\omega_C)$ via the canonical decomposition \eqref{canonical}. Applying \eqref{canonical} and using $\omega_E\simeq\mathcal O_E$, we obtain
$$
h_*\omega_C=\mathcal O_E\oplus\bigoplus_{i=1}^{p-1}L^{i}\!\left(-\left\lfloor\frac{i}{p}B\right\rfloor\right).
$$
The first summand corresponds to the invariant differentials, namely pullbacks of differentials from $E$. Thus the part of $H^0(C,\omega_C)$ transforming by non--trivial characters is
$$
H^0(C,\omega_C)^-=\bigoplus_{i=1}^{p-1}H^0\!\left(E,L^{i}\!\left(-\left\lfloor\frac{i}{p}B\right\rfloor\right)
\right).
$$
Setting
$$
M_i:=L^{i}\!\left(-\left\lfloor\frac{i}{p}B\right\rfloor\right),
$$
we can write
$$
H^0(C,\omega_C)^-
=
\bigoplus_{i=1}^{p-1} H^0(E,M_i),
\quad\text{and hence}\quad
T_{P(h)}(0)\simeq \left(\bigoplus_{i=1}^{p-1} H^0(E,M_i)\right)^\vee.
$$

The differential of the Prym map $\mathcal P_{p,m}$ is induced by cup product. Dualizing and using Serre duality on $E$, the codifferential identifies with the natural map
$$
d\mathcal P_{p,m}(E,B,L)^\vee:
\Sym^2 H^0(C,\omega_C)^-
\longrightarrow
H^0(E,\mathcal O_E(D)).
$$
Using the eigenspace decomposition above, the symmetric square decomposes into blocks
$$
H^0(E,M_i)\otimes H^0(E,M_j).
$$
Since $\mathcal R(p,m)$ parametrizes deformations preserving the cyclic covering structure, the codifferential factors through the $\mathbb Z_p$--invariant part of the multiplication map. This invariant part is obtained by pairing complementary characters, hence only the blocks with $(i,j)=(i,p-i)$ contribute. Setting
$$
N_i:=M_{p-i}=L^{p-i}\!\left(-\left\lfloor\frac{p-i}{p}B\right\rfloor\right),
$$
the codifferential is the direct sum of the multiplication maps
$$
\mu_i:\;
H^0(E,M_i)\otimes H^0(E,N_i)
\longrightarrow
H^0(E,M_i\otimes N_i),
\quad 1\le i\le \frac{p-1}{2}.
$$
For every $j$, neither $ic_j/p$ nor $(p-i)c_j/p$ is an integer, and their sum is $c_j$. Hence
$$
\left\lfloor\frac{ic_j}{p}\right\rfloor
+\left\lfloor\frac{(p-i)c_j}{p}\right\rfloor=c_j-1.
$$
Since $B=\sum_{j=1}^m c_jx_j$, $D=\sum_{j=1}^m x_j$, and $L^p\simeq\mathcal O_E(B)$, it follows that
$$
M_i\otimes N_i
\simeq L^p\left(-\sum_{j=1}^m(c_j-1)x_j\right) \simeq L^p(-B+D)
\simeq\mathcal O_E(D).
$$
Taking degrees gives
\begin{equation}\label{eq:complementary_degrees}
\deg M_i+\deg N_i=m.
\end{equation}
Moreover, the definition of $M_i$, together with $p\deg L=\sum_jc_j$ and $p\nmid ic_j$, shows that $0<\deg M_i<m$. In particular, the target of every multiplication map is $H^0(E,\mathcal O_E(D))$.

We record the following elementary combinatorial fact, whose separate cases will be used in the proof of the finiteness theorem.

\begin{lemma}\label{lem:combinatorial_degrees}
Let $p$ be an odd prime and let $c_1,\dots,c_m\in\{1,\dots,p-1\}$ satisfy $\sum_jc_j\equiv0\pmod p$. Then the following assertions hold.
\begin{enumerate}
\item[(a)] If $m\ge5$, there exists $i$ such that both $M_i$ and $N_i$ have degree at least $2$, and one of them has degree at least $3$.
\item[(b)] If $m=4$, there exists $i$ such that $\deg M_i=\deg N_i=2$.
\item[(c)] If $m=2$, then $\deg M_i=\deg N_i=1$ for every $i$, while if $m=3$, then $\{\deg M_i,\deg N_i\}=\{1,2\}$ for every $i$.
\end{enumerate}
\end{lemma}

\begin{proof}
Assume first that $m\ge4$. Since $1\le c_j\le p-1$, we have $\lfloor B/p\rfloor=0$, hence $M_1=L$. Moreover, $p\deg L=\sum_jc_j$ implies that $1\le\deg L\le m-1$. We distinguish three cases.

\begin{itemize}
\item If $2\le\deg L\le m-2$, the result follows immediately from $M_1=L$ and \eqref{eq:complementary_degrees}.

\item Suppose that $\deg L=1$. Then
$$
\deg M_i=i-\sum_{j=1}^m\left\lfloor\frac{ic_j}{p}\right\rfloor.
$$
Since each term $\lfloor ic_j/p\rfloor$ is non-decreasing in $i$, we have $\deg M_{i+1}\le\deg M_i+1$. The sequence $\deg M_1,\dots,\deg M_{(p-1)/2}$ starts at $1$. Its first change cannot be a decrease, because all degrees are positive, and it cannot remain constant: otherwise \eqref{eq:complementary_degrees} would give $\deg M_{(p+1)/2}=m-1$, a jump greater than $1$. Hence the first change is from $1$ to $2$, and the complementary bundle has degree $m-2$ by \eqref{eq:complementary_degrees}.

\item If $\deg L=m-1$, the complementary line bundle $M_{p-1}$ has degree $1$ by \eqref{eq:complementary_degrees}. Replacing the generator of the cyclic group by its inverse exchanges the indices $i$ and $p-i$, so this case reduces to the preceding one.
\end{itemize}

For $m=4$ this proves (b), while for $m\ge5$ it proves (a). Finally, if $m=2$, positivity and \eqref{eq:complementary_degrees} give $\deg M_i=\deg N_i=1$ for every $i$. If $m=3$, the same argument gives $\{\deg M_i,\deg N_i\}=\{1,2\}$ for every $i$. This proves (c).
\end{proof}

\begin{theorem}\label{thm:surj_codiff_Prym}
Let $p$ be an odd prime and let $h:C\to E$ be a cyclic cover of degree $p$ of an elliptic curve, totally ramified at $m$ points, with local monodromies $a^{c_1},\dots,a^{c_m}$ satisfying $\sum_jc_j\equiv0\pmod p$. Assume that either $p\ge 5$ and $m\ge 2$ or $p=3$ and $m\ge 4$. Then the codifferential of the Prym map $\mathcal P_{p,m}$ is surjective at a general such cover $h$.
\end{theorem}

\begin{proof}
As explained above, the codifferential of the Prym map is the direct sum of the multiplication maps
$$
\mu_i:\;
H^0(E,M_i)\otimes H^0(E,N_i)
\longrightarrow
H^0(E,M_i\otimes N_i)\simeq H^0(E,\mathcal O_E(D)),
\quad 1\le i\le \frac{p-1}{2}.
$$
Since the target space does not depend on $i$, surjectivity amounts to proving that the subspace generated by the images $\ima(\mu_i)$ coincides with $H^0(E,\mathcal O_E(D))$.

Assume first that $m\ge 5$. By Lemma \ref{lem:combinatorial_degrees}, there exists an index $i$ such that one of $M_i,N_i$ has degree at least $3$ and the other has degree at least $2$. Since $E$ is elliptic, every line bundle of positive degree has vanishing $H^1$, and any line bundle of degree at least $2$ is base--point--free. Hence the multiplication map
$$
H^0(E,M_i)\otimes H^0(E,N_i)\longrightarrow H^0(E,M_i\otimes N_i)
$$
is surjective by Mumford's cohomological criterion for multiplication of sections \cite[Theorem 4]{Mumford}. Therefore the codifferential is surjective for $m\ge 5$.

If $m=4$, Lemma \ref{lem:combinatorial_degrees} gives an index $i$ such that $\deg(M_i)=\deg(N_i)=2$. Line bundles of degree $2$ on an elliptic curve are base--point--free pencils, and the base--point--free pencil trick yields the exact sequence
$$
0\to H^0(E,N_i\otimes M_i^{-1})
\to H^0(E,M_i)\otimes H^0(E,N_i)
\longrightarrow H^0(E,M_i\otimes N_i)
\to H^1(E,N_i\otimes M_i^{-1})
\to 0.
$$
For a general cover the degree--zero bundle $N_i\otimes M_i^{-1}$ is non--trivial, hence both outer terms vanish and $\mu_i$ is an isomorphism. Thus the codifferential is generically surjective when $m=4$.

Assume now that $m=3$ and $p\ge5$. By Lemma \ref{lem:combinatorial_degrees}, each complementary pair contains a degree--one line bundle. Since there are at least two complementary pairs, for a general cover we may choose two of these degree--one line bundles which are non--isomorphic, say $M\simeq\mathcal O_E(y)$ and $M'\simeq\mathcal O_E(y')$ with $y\neq y'$. Indeed, the equality of two such line bundles imposes a non--trivial relation in $\Pic^0(E)$ and hence defines a proper closed locus in the space of branch data.

Since a degree--one line bundle has a unique section up to scalar, multiplication by this section identifies the corresponding image with
$$
H^0(E,\mathcal O_E(D-y))
\subset H^0(E,\mathcal O_E(D)).
$$
The other chosen line bundle gives $H^0(E,\mathcal O_E(D-y'))$. Since $\deg(D)=3$, the line bundle $\mathcal O_E(D)$ is very ample and therefore separates the points $y$ and $y'$. Hence
$$
H^0(E,\mathcal O_E(D-y))
\neq
H^0(E,\mathcal O_E(D-y')).
$$
These are distinct hyperplanes in the three-dimensional space $H^0(E,\mathcal O_E(D))$, and consequently
$$
H^0(E,\mathcal O_E(D-y))+H^0(E,\mathcal O_E(D-y'))
=
H^0(E,\mathcal O_E(D)).
$$
Thus the codifferential is generically surjective for $m=3$ whenever $p\ge5$.

Finally, assume that $m=2$ and $p\ge5$. Up to changing the generator of the cyclic group, we may write
$$
B=x_1+(p-1)x_2.
$$
Then
$$
M_1=L,\qquad M_2=L^2(-x_2),
$$
and these line bundles belong to distinct complementary pairs. Suppose that the image lines of $\mu_1$ and $\mu_2$ coincide. Since all the factors have degree one, equality of the divisors of their generating sections implies either
$$
M_1\simeq M_2
\qquad\text{or}\qquad
M_1\otimes M_2\simeq\mathcal O_E(D).
$$
In the first case, $L\simeq\mathcal O_E(x_2)$, and comparison of $L^p$ with $\mathcal O_E(B)$ gives $x_1=x_2$, a contradiction. In the second case,
$$
L^3\simeq\mathcal O_E(x_1+2x_2).
$$
Taking $p$-th powers and using
$$
L^p\simeq\mathcal O_E(x_1+(p-1)x_2)
$$
gives
$$
\mathcal O_E\bigl((p-3)(x_1-x_2)\bigr)\simeq\mathcal O_E.
$$
This is a proper closed condition on the branch data. Hence, for a general cover, the image lines of $\mu_1$ and $\mu_2$ are distinct, and therefore generate $H^0(E,\mathcal O_E(D))$.
\end{proof}

\begin{corollary}\label{cor:generic_finite_PT_map}
Assume either $p\ge5$ and $m\ge2$, or $p=3$ and $m\ge4$, and let
$$
\mathcal{PT}_{p,m}:\mathcal R(p,m)^{\xi}\longrightarrow
\mathcal A_{\frac{m}{2}(p-1)}
$$
be the Prym--Tyurin map defined above. Then $\mathcal{PT}_{p,m}$ is generically finite onto its image.
\end{corollary}

\begin{proof}
Since $\mathcal R(p,m)^{\xi}\to\mathcal R(p,m)$ is finite, the map $\mathcal{PT}_{p,m}$ is generically finite onto its image if and only if the Prym map
$$
\mathcal P_{p,m}:\mathcal R(p,m)\longrightarrow \mathcal A^\delta_{\frac{m}{2}(p-1)},
\quad \delta=(1,\dots,1,p),
$$
is generically finite onto its image. By Theorem \ref{thm:surj_codiff_Prym}, the codifferential of $\mathcal P_{p,m}$ is surjective at a general point whenever $p\ge5$ and $m\ge2$, or $p=3$ and $m\ge4$. Hence the differential of $\mathcal P_{p,m}$ is injective at a general point, and therefore $\mathcal P_{p,m}$ is generically finite onto its image. The same conclusion holds for $\mathcal{PT}_{p,m}$.
\end{proof}

\begin{proposition}\label{prop:generic_PT_not_Jac}
Assume that either $p\ge5$ and $m\ge2$, or $p=3$ and $m\ge4$. Then a general Prym--Tyurin variety in the family defined by $\mathcal{PT}_{p,m}$ is not isomorphic as a ppav  to  a Jacobian of a smooth curve.
\end{proposition}

\begin{proof}
Our family $\mathcal P$ has dimension $m$, since the Prym--Tyurin map is generically finite. Fix a generic $P\in\mathcal P$. By construction, $P$ carries an automorphism $\sigma$ of order $p$ whose action on $H^0(P,\Omega^1_P)$ has no trivial part. If $P\simeq JC$ for some smooth curve $C$ of genus $g=\dim P$, then by Torelli $\sigma$ is induced by an automorphism of $C$, still denoted $\sigma$, of order $p$. Since
$$
H^0(C,\omega_C)^{\langle\sigma\rangle}\;=\;H^0(P,\Omega^1_P)^{\langle\sigma\rangle}\;=\;0,
$$
the quotient $C/\langle\sigma\rangle$ has genus $0$, hence $C$ is a cyclic degree $p$ cover of $\mathbb P^1$.

In our situation one has
$$
g=\dim P=\frac{m}{2}(p-1).
$$
Let $\pi:C\to\mathbb P^1$ be the quotient map, and let $r$ denote the number of branch points of $\pi$ (with total ramification, as in the cyclic case). Then by the Riemann--Hurwitz formula one concludes that $r=m+2$. Therefore the moduli of such cyclic $p$--covers has dimension $r-3=m-1$.

Consequently, the locus of Jacobians that could arise this way has dimension strictly smaller than $\dim\mathcal P=m$. Hence a generic $P\in\mathcal P$ cannot be a Jacobian.
\end{proof}

\subsection{The Abel--Prym map}\label{subsec:abel_prym}

Fix a curve $\widetilde C$ as in Theorem \ref{Theo1} and let $g\in X\setminus\{a\}$. Set
$$
P_g:=k_g^*P(h_g)\subset J\widetilde C.
$$
By Theorem \ref{Theo1}, $P_g$ is a Prym--Tyurin variety of exponent $p$ for $\widetilde C$, endowed with the polarization induced by $\widetilde\Theta$. Following \cite{Coelho2021}, the Abel--Prym map associated with an abelian subvariety $A\subset J\widetilde C$ is the morphism
$$
\varphi_A:=N_A\circ \alpha_{\widetilde C}:\widetilde C\longrightarrow A,
$$
where $\alpha_{\widetilde C}$ is the Abel map and $N_A$ is the norm endomorphism of $A$. We denote by $\varphi_{P_g}$ the Abel--Prym map of $P_g$.

\begin{proposition}\label{prop:degree_abel_prym}
For every $g\in X\setminus\{a\}$, the Abel--Prym map
$$\varphi_{P_g}:\widetilde C\longrightarrow P_g$$
is finite onto its image and has degree $p$. Moreover, it factors through the quotient $k_g:\widetilde C\to C_g$.
\end{proposition}

\begin{proof}
Fix $g\in X\setminus\{a\}$ and write $g=a^{-j}b$ for a unique $j\in\{0,\dots,p-1\}$. By the proof of Proposition \ref{prop:iso_decomposition}, $P_g$ is the isotypical component of $J\widetilde C$ corresponding to the rational representation $W_{\omega\otimes\omega^{j}}$ of $G_p$. Since
$$\ker(\omega\otimes\omega^{j})=\langle a^{-j}b\rangle=\langle a^{p-j}b\rangle,$$
the kernel of $W_{\omega\otimes\omega^{j}}$ is $\langle g\rangle$. By \cite[Theorem~9]{Coelho2021}, there is a factorization
$$\varphi_{P_g}=\bar\varphi_{P_g}\circ k_g,$$
where $\bar\varphi_{P_g}:C_g\to P_g$ has image $\varphi_{P_g}(\widetilde C)$. Since $P_g$ is a Prym--Tyurin variety of exponent $p$ for $\widetilde C$ and $|\langle g\rangle|=p$, \cite[Theorem~11(i)]{Coelho2021} implies that $\bar\varphi_{P_g}$ is the normalization of its image and $\deg(\varphi_{P_g})=p$.
\end{proof}

Let $\varphi_{P(h_g)}:=N_{P(h_g)}\circ\alpha_{C_g}:C_g\to P(h_g)$ be the Abel--Prym map associated with $P(h_g)\subset JC_g$, and set $u_g:=k_g^*|_{P(h_g)}:P(h_g)\to P_g$. By the proof of Theorem~\ref{Theo1}, $\ker u_g=\langle\eta_g\rangle\subset P(h_g)[p]$. Hence there is an isogeny $\mu_g:P_g\to P(h_g)$ such that
$$\mu_g\circ u_g=[p]_{P(h_g)},\qquad u_g\circ\mu_g=[p]_{P_g}.$$
Let $\bar\varphi_{P_g}^*:JC_g\to P_g$ be the homomorphism induced by $\bar\varphi_{P_g}$. The factorization of $\varphi_{P_g}$ and the identity $\Nm_{k_g}\circ\alpha_{\widetilde C}=\alpha_{C_g}\circ k_g$ give $N_{P_g}=\bar\varphi_{P_g}^*\circ\Nm_{k_g}$. Moreover,
$$N_{P_g}\circ k_g^*=k_g^*\circ N_{P(h_g)},$$
since both sides vanish on $h_g^*JE$ and coincide on $P(h_g)$, which generate $JC_g$ up to isogeny. Using $\Nm_{k_g}\circ k_g^*=[p]_{JC_g}$, we obtain
$$[p]_{P_g}\circ\bar\varphi_{P_g}^*=k_g^*\circ N_{P(h_g)}.$$
Composing with $\mu_g$ and cancelling multiplication by $p$ in $\Hom(JC_g,P(h_g))$ gives $\mu_g\circ\bar\varphi_{P_g}^*=N_{P(h_g)}$. Consequently, for compatible base points, we have the commutative diagram
$$
\begin{tikzcd}[column sep=large]
\widetilde C \arrow[r,"\varphi_{P_g}"] \arrow[d,"k_g"'] & P_g \arrow[d,"\mu_g"]\\
C_g \arrow[r,"\varphi_{P(h_g)}"'] & P(h_g).
\end{tikzcd}
$$
For arbitrary base points, the diagram commutes up to translation. This allows us to obtain the following result.

\begin{proposition}\label{prop:generic_embedding}
Assume $m\ge 2$. Then the morphism
$$\bar\varphi_{P_g}:C_g\longrightarrow P_g$$
is birational onto its image and immersive at a general point of $C_g$. Equivalently, it restricts to an embedding on a nonempty open subset of $C_g$.
\end{proposition}

\begin{proof}
By Section \ref{section5} we have
$$H^0(C_g,\omega_{C_g})^-=\bigoplus_{i=1}^{p-1} H^0(E,M_i),\qquad M_i=L^{i}\!\left(-\left\lfloor\frac{i}{p}B\right\rfloor\right).$$
The codifferential of $\varphi_{P(h_g)}$ identifies $H^0(P(h_g),\Omega^1_{P(h_g)})$ with $H^0(C_g,\omega_{C_g})^-$. Since $\mu_g$ is an isogeny, its codifferential is an isomorphism, and the identity $\varphi_{P(h_g)}=\mu_g\circ\bar\varphi_{P_g}$ gives
$$\ima\bigl(d\bar\varphi_{P_g}^{\vee}\bigr)=H^0(C_g,\omega_{C_g})^-.$$
Thus the differential of $\bar\varphi_{P_g}$ vanishes at $x\in C_g$ if and only if $x$ is a base point of this subspace.

Let $x\in C_g$ be a general point outside the ramification locus of $h_g$ and set $y=h_g(x)\in E$. If $m\ge3$, there exists an index $i$ such that $\deg(M_i)\ge2$. Since $M_i$ is base--point--free, there is a section $s\in H^0(E,M_i)$ with $s(y)\neq0$. The corresponding differential in $H^0(C_g,\omega_{C_g})^-$ does not vanish at $x$.

If $m=2$, then $\deg(M_i)=1$ for every $i$, and each $H^0(E,M_i)$ is generated by a section vanishing at a unique point $y_i\in E$. Away from the ramification locus, the corresponding differential vanishes only on $h_g^{-1}(y_i)$. Hence a general point of $C_g$ is not a base point of $H^0(C_g,\omega_{C_g})^-$.

Therefore $\bar\varphi_{P_g}$ is immersive at a general point. Since it is the normalization of its image, it restricts to an embedding on a nonempty open subset of $C_g$.
\end{proof}

\begin{corollary}
Assume $m\ge2$. Then the Abel--Prym map $\varphi_{P_g}:\widetilde C\to\varphi_{P_g}(\widetilde C)$ is generically an \'etale covering of degree $p$, and $C_g$ is the normalization of the Abel--Prym curve $\varphi_{P_g}(\widetilde C)\subset P_g$.
\end{corollary}

\begin{proof}
By Proposition~\ref{prop:degree_abel_prym}, the map $\varphi_{P_g}$ factors as $\bar\varphi_{P_g}\circ k_g$, where $k_g:\widetilde C\to C_g$ is an \'etale covering of degree $p$ and $\bar\varphi_{P_g}$ is the normalization of its image. The result follows from Proposition~\ref{prop:generic_embedding}.
\end{proof}

\bibliographystyle{amsalpha-inits}
\bibliography{references}

\end{document}